\documentclass[12pt, reqno]{amsart}
\usepackage{amsfonts}
\usepackage{txfonts}
\usepackage{}
\usepackage{pdfpages}

\numberwithin{equation}{section}

\usepackage{amsmath}
\usepackage{amssymb}
\usepackage{amstext}
\usepackage{amsxtra}
\usepackage{amscd}
\usepackage{mathrsfs}
\usepackage{fancyhdr}

\newcommand{\F}{\mathbb{F}}

\newtheorem{theorem}{Theorem}[section]
\newtheorem{lemma}[theorem]{Lemma}
\newtheorem{proposition}[theorem]{Proposition}

\newtheorem{conjecture}[theorem]{Conjecture}

\newcommand{\proofbox}{\hspace*{.2in}\raisebox{3pt}
{\fbox{\rule{0pt}{.7pt}\rule{.7pt}{0pt}}}}

\def\whitebox{{\hbox{\hskip 1pt
 \vrule height 6pt depth 1.5pt
 \lower 1.5pt\vbox to 7.5pt{\hrule width
    3.2pt\vfill\hrule width 3.2pt}%
 \vrule height 6pt depth 1.5pt
 \hskip 1pt } }}
\def\qed{\ifhmode\allowbreak\else\nobreak\fi\hfill\quad\nobreak
     \proofbox\medbreak}
\begin{document}

\title[Strongly Regular Cayley Graphs]
{Constructions of Strongly Regular Cayley Graphs Using Index Four Gauss Sums}

\author[Ge, Xiang and Yuan]{Gennian Ge$^*$, Qing Xiang$^{**}$, and Tao Yuan}

\thanks{$^*$ Research supported by the National Outstanding Youth Science  Foundation of China under Grant No.~10825103, National Natural
Science Foundation of China under Grant No.~61171198,  and Specialized Research Fund for the Doctoral Program of Higher Education.}
\thanks{$^{**}$Research supported in part by NSF Grant DMS 1001557, by the
Overseas Cooperation Fund (grant 10928101) of China, and by Y. C. Tang disciplinary development fund, Zhejiang University.}

\address{Department of Mathematics, Zhejiang University, Hangzhou 310027, Zhejiang, P. R. China}
\email{gnge@zju.edu.cn}

\address{Department of Mathematical Sciences, University of Delaware, Newark, DE 19716, USA} \email{xiang@math.udel.edu}

\address{Department of Mathematics, Zhejiang University, Hangzhou 310027, Zhejiang, P. R. China}
\email{matheufreedom@gmail.com}


\begin{abstract} We give a construction of strongly regular Cayley graphs on finite fields $\F_q$ by using union of cyclotomic classes and index $4$ Gauss sums. In particular, we obtain two infinite families of strongly regular graphs with new parameters.
\end{abstract}

\maketitle

\section{Introduction}
A {\it strongly regular graph} srg$(v,k,\lambda,\mu)$ is a simple and undirected graph, neither complete nor edgeless, that has the following properties:

(1) It is a regular graph of order $v$ and valency $k$.

(2) For each pair of adjacent vertices $x,y$, there are $\lambda$ vertices adjacent to both $x$ and $y$.

(3) For each pair of nonadjacent vertices $x,y$, there are $\mu$ vertices adjacent to both $x$ and $y$.

For example, a pentagon is an srg$(5,2,0,1)$,  the $3\times 3$ grid (the Cartesian product of two triangles) is an srg$(9,4,1,2)$, and the Petersen graph is an srg$(10,3,0,1)$. The first two examples can be generalized. Let $q=4t+1$ be a prime power. The Paley graph P($q$) is the graph with the elements of the finite field $\mathbb{F}_q$ as vertices; two vertices are adjacent if and only if their difference is  a nonzero square in $\mathbb{F}_q$. One can readily check that P($q$) is an srg$(4t+1,2t,t-1,t)$. For a survey on strongly regular graphs, we refer the reader to \cite{bh} and \cite{cg}. Strongly regular graphs are closely related to two-weight linear codes, projective two-intersection sets in finite geometry, quasi-symmetric designs, and partial difference sets. We refer the reader to \cite{CK86, Ma94, bh, cg} for these connections.

The adjacency matrix of a (simple) graph $\Gamma$ is a $(0,1)$-matrix $A$ with rows and columns both indexed by the vertices of $\Gamma$, where $A_{xy}=1$ if and only if $x,y$ have an edge in $\Gamma$. Clearly $A$ is symmetric with zeros on the diagonal. The eigenvalues of $\Gamma$ are by definition the eigenvalues of its adjacency matrix $A$. For convenience, we call an eigenvalue of $\Gamma$ { \it restricted } if it has an eigenvector orthogonal to the all-one vector. Below is a well-known characterization of srg by using their eigenvalues; we refer the reader to \cite{bh} for its proof.

\begin{theorem}\label{basicthm}
For a graph $\Gamma$ of order $v$, neither complete nor edgeless, with adjacency matrix $A$, the following are equivalent:\\
\quad (1) $\Gamma$ is an srg$(v,k,\lambda,\mu)$ for certain integers $ k,\lambda,\mu$.\\
\quad (2) $A^2=(\lambda-\mu)A+(k-\mu)I +\mu J$, where I, J are the identity matrix and the all-one matrix, respectively.\\
\quad (3) A has precisely two distinct restricted eigenvalues.
\end{theorem}
The two distinct restricted eigenvalues of an srg are usually denoted by $r$ and $s$, where $r$ is the positive eigenvalue and $s$ the negative one. The Paley graphs are probably the simplest examples of the so-called cyclotomic strongly regular graphs, which we define below. Let $\mathbb{F}_{p^f}$ be  the finite field of order $p^f$, where $p$ is a prime and $f$ is a positive integer. Let $D$ be a subset of $\mathbb{F}_{p^f}$ such that $-D=D$ and $0 \not \in D$. We define the  { \it Cayley graph} $Cay(\mathbb{F}_{p^f},D)$ to be the graph with the elements of $\mathbb{F}_{p^f}$ as vertices; two vertices are adjacent if and only if their difference belongs to $D$. When $D$ is a subgroup of the multiplicative group $\mathbb{F}_{p^f}^*$ of  $\mathbb{F}_{p^f}$ and  $Cay(\mathbb{F}_{p^f},D)$ is strongly regular, then we say that $Cay(\mathbb{F}_{p^f},D)$ is a {\it cyclotomic strongly regular graph}. Specializing to the case where $D$ is the subgroup of $\mathbb{F}_{q}^*$ consisting of the nonzero squares, where $q$ is a prime power congruent to 1 modulo 4,  we see that $Cay(\mathbb{F}_{q},D)$ is nothing but the Paley graph P($q$).

Cyclotomic srg have been extensively studied by many authors; see \cite{McE, vanLintSch, bwx, Lang, batdov, schwh, ikutam}. Some of these authors used the language of cyclic codes in their investigations. We choose to use the language of srg. Let $D$ be a subgroup of $\mathbb{F}_{p^f}^*$ of index $N>1$. If $D$ is the multiplicative group of a subfield of $\mathbb{F}_{p^f}$, then it is easy to show that   $Cay(\mathbb{F}_{p^f},D)$ is an srg. These cyclotomic srg are usually called {\it subfield examples}. Next if there exists a positive integer $t$ such that $p^t\equiv -1$ (mod $N$), then $Cay(\mathbb{F}_{p^f},D)$ is an srg by an old result of Stickelberger \cite{stic}. These examples are usually called {\it semi-primitive} cyclotomic srg. The following conjecture of Schmidt and White \cite{schwh} says that besides the two classes of cyclotomic srg mentioned above, there are only 11 sporadic examples of cyclotomic srg.

\begin{conjecture}{\em (Conjecture 4.4, \cite{schwh})} Let $\F_{p^f}$ be the finite field of order $p^f$, $N|(\frac{p^f-1}{p-1})$, $N>1$, and let $C_0$ be the subgroup of $\F_{p^f}^*$ of index $N$. Assume that $-C_0=C_0$. If ${\rm Cay}(\F_{p^f}, C_0)$ is an srg, then one of the following holds:\\

\quad (1) (subfield case) $C_0=\F_{p^e}^*$, where $e|f$,\\

\quad (2)  (semi-primitive case) There exists a positive integer $t$ such that $p^t\equiv -1$ (mod $N$),\\

\quad (3) (exceptional case)  ${\rm Cay}(\F_{p^f}, C_0)$ is one of the eleven ``sporadic" examples appearing in the following table.\\

$$\begin{tabular}{ccccc}
%\multicolumn{3}{c}
\hline
 & N & p & f & $[(\mathbb{Z}_N)^*:\langle p\rangle]$\\ \hline
  &11 & 3& 5 &2 \\
 &19 &5&  9 &2\\
 &35& 3& 12 &2\\
 &37& 7& 9& 4\\
 &43 &11& 7&6\\
 &67 &17& 33& 2\\
 &107 &3 &53& 2\\
 &133 &5 &18& 6  \\
 &163 &41& 81& 2\\
 &323 &3& 144&2\\
 &499 & 5 & 249& 2\\\hline
 \multicolumn{3}{c}{\quad\quad\quad\quad\quad{\bf Table \quad I}}\\\\
\end{tabular}$$
\end{conjecture}

The above conjecture remains open. On the construction side, semi-primitive Gauss sums have been quite useful for constructing strongly regular Cayley graphs. Here by {\it semi-primitive Gauss sums} $g(\chi)$ over $\F_{p^f}$, where the order of $\chi$ is $N$, we mean that there exists some positive integer $t$ such that $p^t\equiv -1\pmod N$. In such a situation, it is known that an arbitrary union of cyclotomic classes of order $N$ of $\F_{p^f}$ will give rise to an srg. We refer the reader to \cite{vanLintSch}, \cite{bmw}, \cite{bwx} and \cite{vanDamMu} for work in this direction.  Quite recently,   motivated by the examples of De Lange \cite{del} and Ikuta and Munemasa \cite{ikutam}, Feng and Xiang \cite{fengxiang} considered the problem of constructing strongly regular graphs $Cay(\mathbb{F}_{p^f},D)$, where $D$ is a union of at least two cyclotomic classes of order $N$ and it is assumed that a single cyclotomic class of order $N$ does not give rise to an srg. They succeeded in generalizing seven of the index 2 examples of cyclotomic srg in Table I into infinite families. The main tools used in \cite{fengxiang} are index 2 Gauss sums. We remark that even though the first example in Table I is an index 2 example (${\rm ord}_{11}(3)=5$), the construction in \cite{fengxiang} could not generalize it into an infinite family since ${\rm ord}_{11^m}(3)\neq \phi(11^m)/2$ when $m>1$.

In this paper, we use similar idea to construct strongly regular Cayley graphs. Our goal is to generalize the index 4 example in Table I. Naturally the main tools that we use are index 4 Gauss sums, which will be introduced Section 2. We obtain two infinite families of srg with new parameters. The first family generalizes the index 4 example listed in Table I, and it has parameters
$$v=7^{9\cdot 37^{m-1}}, \ k=\frac{v-1}{37}, \; r=\frac{9\cdot 7^{\frac{ 9\cdot 37^{m-1}  -1} {2} } -1}{37},\; {\rm and}\ s=\frac{-4 \cdot 7^{\frac{9\cdot 37^{m-1} +1} {2} } -1}{37},$$
where $m\geq 1$ is an integer. (Note that the $\lambda$ and $\mu$ values of the srg can be computed from $v,k,r$ and $s$.) The second family generalizes a (trivial) subfield example of cyclotomic srg, and it has parameters
$$v=3^{3\cdot 13^{m-1}},\ k=\frac{v-1}{13}, \; r=\frac{3^{ \frac{3\cdot 13^{m-1}+3 } {2}   }-1  }  {13},  \; {\rm and} \; s=\frac{-4\cdot 3^{ \frac{3\cdot 13^{m-1}-1 } {2}   }-1  }  {13},$$
where $m\geq 1$ is an integer.

\section{Index 4 Gauss sums}
Let $p$ be a prime, $f$ be a positive integer, and $q=p^f.$ Let $\mathbb{F}_q$ be the finite field of order $q$,\  $\zeta_p$ be a complex primitive $p$-th root of unity, and ${\rm Tr}_{q/p}$ be the trace from $\mathbb{F}_q$ to $\mathbb{F}_p$. The multiplicative characters of $\mathbb{F}_q$ are the homomorphisms from the multiplicative group $\mathbb{F}_q^*$ to the multiplicative group $\mathbb{C}^*$ of the complex field $\mathbb{C}$. On the other hand, the additive characters of $\mathbb{F}_q$ are the homomorphisms from the additive group $(\mathbb{F}_q,+)$ to $\mathbb{C}^*$, and they are given by
\begin{equation*}
\psi_a: \mathbb{F}_q \rightarrow \mathbb{C}^*, \quad \psi_a(x)=\zeta_p^{{\rm Tr}_{q/p}(ax)},
\end{equation*}
where $a\in \mathbb{F}_q$. We usually write $\psi_1$ simply as $\psi$, which is called the {\it canonical} additive character of $\mathbb{F}_q$.

Now let $\chi$ be a multiplicative character of $\mathbb{F}_q$. Define the Gauss sum by
\begin{equation*}
g(\chi)=\sum_{x \in \mathbb{F}_q^*} \chi(x)\psi(x).
\end{equation*}

We first list some basic properties of Gauss sums.

\begin{proposition}\label{basicprop}
{\em (Lemma 1.1 \cite{Feng})} (1) Let $\chi_0$ be the trivial multiplicative character of $\mathbb{F}_q$. Then $g(\chi_0)=-1$. Also $g(\chi)\overline{g(\chi)}=q$ for any $\chi\neq \chi_0$.

(2) Let $N|(q-1)$, $\chi$ be a multiplicative character of $\mathbb{F}_q$ of order $N$, and $\sigma_{a,b}\in Gal(\mathbb{Q}(\zeta_{N},\zeta_p)/\mathbb{Q})$ be such that $\sigma_{a,b}(\zeta_N)=\zeta_N^a$ and $\sigma_{a,b}(\zeta_p)=\zeta_p^b$. Then $\sigma_{a,b}(g(\chi))=\overline{\chi}^a(b)g(\chi^a)$. Also $\sigma_{p,1}(g(\chi))=g(\chi^p)=g(\chi)$.

%(3) $\overline{g(\chi)}=\chi(-1)g(\chi)$ for any $\chi\neq \chi_0$.
%(4) For a multiplicative character $\chi$ of order $N$, we have $g(\chi)^N\in \mathbb{Z}[\zeta_N]$.
\end{proposition}

For more properties of Gauss sums, we refer the reader to \cite{bew} and \cite{irero}. Gauss sums can be viewed as the Fourier coefficients of the Fourier expansion of the additive characters in terms of the multiplicative characters of $\mathbb{F}_q$. That is,
\begin{equation}\label{inversion}
\psi(a)= \frac{1}{q-1}\sum_{\chi \in \widehat{\mathbb{F}_q^*} } g(\bar{\chi})\chi(a), \quad \mbox{for all $a \in \mathbb{F}_q^*$},
\end{equation}
where $\bar{\chi}=\chi^{-1}$ and $\widehat{\mathbb{F}_q^*} $ denotes the character group of $\mathbb{F}_q^*$.

In this paper, we will need certain index 4 Gauss sums, which we define below.

Let $p$ be a prime, $N \geq 2$ such that $\gcd(p(p-1),N)=1$. Thus $p\in \mathbb{Z}_N^*$, the unit group of $\mathbb{Z}_N$. Furthermore, we assume that $-1 \not \in \langle p \rangle$ and the order of $p$ modulo $N$ is $f=\frac{\phi(N)}{4}$. It follows that $[\mathbb{Z}_N^* : \langle p \rangle ]=4$ and the decomposition field $K$ of $p$ in the cyclotomic field $\mathbb{Q}(\zeta_N)$ is a quartic abelian imaginary field. Let $\chi$ be a multiplicative character of $\mathbb{F}_q$ of order $N$. Then the Gauss sum $g(\chi)$ is called an {\it index 4 Gauss sum}. Note that since we assumed that $\gcd(N, p-1)=1$, we have $\chi(b)=1$ for any $b\in\mathbb{F}_p^*$, where $\chi\in \widehat{\mathbb{F}_q^*} $ has order $N$. It follows that $g(\chi)\in\mathbb{Z}[\zeta_N]$ by part (2) of Proposition~\ref{basicprop}.

Since $\gcd(p(p-1),N)=1$, $N$ must be odd. The assumption $[\mathbb{Z}_N^* : \langle p \rangle ]=4$ implies that $N$ has at most three distinct prime factors (cf. \cite{Feng}).  In fact, the authors of \cite{Feng} listed all possibilities of $N$ satisfying the above assumptions. In this paper, we are only concerned with one of these possibilities, namely, $N=p_1^m$, where $m$ is a positive integer, $p_1$ is an odd prime and $p_1 \equiv 5\ (\bmod \ 8)$. In this case, the decomposition field $K$ is the unique imaginary cyclic quartic subfield of $\mathbb{Q}(\zeta_N)$. In fact, $K$ is a subfield of $\mathbb{Q}(\zeta_{p_1})$. The Galois group $Gal(K/ \mathbb{Q})$ is canonically isomorphic to the group $\mathbb{Z}_N^*/ \langle p \rangle$. Henceforth, we often identify these two groups. We can choose a primitive element $g$ modulo $p_1$ such that $g$ is also a primitive element modulo $N=p_1^m$ (cf. \cite[p.~43]{irero}). Let $\sigma: \zeta_N\mapsto\zeta_N^g$. Then $\sigma$ is a generator of $Gal(\mathbb{Q}(\zeta_N)/\mathbb{Q})$ and its restriction to $K$ is a generator of  $Gal(K/ \mathbb{Q})\cong \mathbb{Z}_N^*/ \langle p \rangle\cong \mathbb{Z}_{p_1}^*/ \langle p \rangle$. By the choice of $g$ and the index 4 assumption we have $\mathbb{Z}_{p_1}^*= \langle p\rangle  \cup g \langle p\rangle \cup g^2\langle p\rangle \cup g^3 \langle p\rangle.$  We will use the following notation.
\\

$\tilde{C_j}= g^j \langle p \rangle \subseteq \mathbb{Z}_{p_1}^* \quad (0\leq j \leq 3); $
\\

$\tilde{f}=\frac{\phi(p_1)}{4}=\frac{p_1-1}{4};$
\\

$b_j= \frac{1}{p_1} \sum _{z \in ([1,p_1-1]\cap \tilde{C}_j)} z \quad (0\leq j \leq 3)$, where $[1,p_1-1]$ denotes the set of integers $x$, $1\leq x\leq p_1-1$;
\\

$b$= min$\{b_0, b_1, b_2, b_3 \} =b_\lambda$ for some $\lambda \in \{0,1,2,3\}$;
\\

$c$=min$\{b_{\lambda+1}-b, b_{\lambda+3}-b\}$, where the subscripts are read modulo 4;
\\

$\eta_j= \sum_{a\in \tilde{C}_j } \zeta _{p_1}^a \quad (0\leq j \leq 3),$ where $\zeta_{p_1}$ is a complex primitive $p_1$-th root of unity.

\begin{lemma}\label{integralbasis}{\em (\cite{Feng})} With the above assumptions and notation $\{\eta_j \mid  0\leq j \leq 3\}$ is an integral basis of K, and $\eta_j =\sigma^j(\eta_0)$, where $\sigma(\zeta_{p_1})=\zeta_{p_1}^g$. The equation $p_1=X^2+Y^2$ has a unique integer solution $(A,B)$ such that $A\equiv 3\ (\bmod \ 4)$. Furthermore,
\begin{eqnarray*}
4\eta_0, 4\eta_2=(-1+\sqrt{p_1})\pm i\sqrt{2}[p_1-A\sqrt{p_1}]^{\frac{1}{2}}, \\
4\eta_1, 4\eta_3=(-1-\sqrt{p_1})\pm i\sqrt{2}[p_1+A\sqrt{p_1}]^{\frac{1}{2}}.
\end{eqnarray*}
\end{lemma}

Below let $\chi$ be a multiplicative character of $\mathbb{F}_q$ of order $N$.
\begin{theorem}\label{divisibility}{\em (\cite{Feng})}
 Under the above assumptions, we have $p^{-\frac{f-\tilde{f}}{2}-b}g(\chi) \in O_K$ (the integer ring of $K$).
\end{theorem}

By Lemma~\ref{integralbasis}, we now write $p^{-\frac{f-\tilde{f}}{2}-b}g(\chi)$ as
\begin{equation*}
p^{-\frac{f-\tilde{f}}{2}-b}g(\chi)=N_0 \eta_0 +N_1 \eta_1+ N_2\eta _2 +N_3 \eta_3, \;  N_i\in\mathbb{Z},\; \forall i.
\end{equation*}
Without loss of generality we assume that
\begin{eqnarray*}
4\eta_0=(-1+\sqrt{p_1})+ i\sqrt{2}[p_1-A\sqrt{p_1}]^{\frac{1}{2}}=4{\overline\eta_2}, \\
4\eta_1=(-1-\sqrt{p_1})+ i\sqrt{2}[p_1+A\sqrt{p_1}]^{\frac{1}{2}}=4{\overline \eta_3}.
\end{eqnarray*}
Then
\begin{equation}\label{expressionforg1}
\begin{array}{lll}
 4 p^{-\frac{f-\tilde{f}}{2}-b}g(\chi) & =-(N_0+N_1+N_2+N_3)+(N_0-N_1+N_2-N_3)\sqrt{p_1}\ + \\
 & i\sqrt{2}[(N_0-N_2)(p_1-A\sqrt{p_1})^{\frac{1}{2}}+ (N_1-N_3)(p_1+A\sqrt{p_1})^{\frac{1}{2}}].
\end{array}
\end{equation}
We make the following transformation:
\begin{equation*}
\left\{\begin{array}{lll}
     M_0=N_0+N_1+N_2+N_3,\\
     M_1=N_0+N_1-N_2-N_3,\\
     M_2=N_0-N_1+N_2-N_3,\\
     M_3=N_0-N_1-N_2+N_3,\\
                 \end{array}
                 \right.
  \quad
  \left\{\begin{array}{lll}
     4N_0=M_0+M_1+M_2+M_3,\\
     4N_1=M_0+M_1-M_2-M_3,\\
     4N_2=M_0-M_1+M_2-M_3,\\
     4N_3=M_0-M_1-M_2+M_3.\\
                 \end{array}
                 \right.
\end{equation*}
Then
\begin{equation}\label{expressionforg2}
\begin{array}{lll}
 4 p^{-\frac{f-\tilde{f}}{2}-b}g(\chi) & =-M_0+M_2\sqrt{p_1}+\\
& i\sqrt{2}[\frac{M_1+M_3}{2}(p_1-A\sqrt{p_1})^{\frac{1}{2}}+\frac{M_1-M_3}{2}(p_1+A\sqrt{p_1})^{\frac{1}{2}}].
\end{array}
\end{equation}

\begin{theorem}{\em (\cite{Feng})}
The integers $M_0,M_1,M_2,M_3$ defined above satisfy the following conditions:
\begin{equation*}
\left \{\begin{array}{lll}
    16 p^{\tilde{f}-2b}=M_0^2+p_1(M_1^2+M_2^2+M_3^2),\\
    2M_0M_2+2AM_1M_3=B(M_1^2-M_3^2),\\
    M_0+M_1+M_2+M_3 \equiv 0 \ (\bmod \ 4),\\
     M_1 \equiv M_2 \equiv M_3 \ (\bmod \ 2),\\
     M_0\equiv 4p^{-b}\ (\bmod p_1). \\
          \end{array} \right.
\end{equation*}
\end{theorem}

\section{Cyclotomic classes and strongly regular Cayley graphs}

Let $q=p^f$ be a prime power, and $\gamma$ be a fixed primitive element of $\mathbb{F}_q$. Let $N>1 $ be a divisor of $q-1$. Then the $N$-th cyclotomic classes $C_0,C_1,\ldots, C_{N-1}$ are defined by
\begin{equation*}  C_i= \{\gamma^{i+jN}\mid 0 \leq j \leq \frac{q-1}{N}-1\},\end{equation*}
where $0 \leq i \leq N-1$.

Note that $C_0$ consists of all the $N$-th powers in $\mathbb{F}_q^*$. Therefore $C_0$ does not depend on the choice of $\gamma$. The other classes $C_i$, $1\leq i\leq N-1$, do depend on the choice of $\gamma$. As usual, let $\psi$ be the canonical additive character of $\mathbb{F}_q$. The $N$-th {\it cyclotomic periods} (also called {\it Gauss periods}) are defined by
\begin{equation*}
\tau_a=\sum_{x \in C_a }\psi(x),
\end{equation*}
where $0 \leq a \leq N-1$.

Now using (\ref{inversion}), we have
\begin{eqnarray*}
\tau_a  & = &\sum_{x \in C_0}\psi(\gamma^a x)\\
& = &\sum_{x \in C_0} \frac{1}{q-1}\sum_{\chi \in \widehat{\mathbb{F}_q^*} } g(\bar{\chi}) \chi(\gamma^a x)\\
& =&\frac{1}{(q-1)}\sum_{\chi \in \widehat{\mathbb{F}_q^*} } g(\bar{\chi})\chi(\gamma^a)\sum_{x \in C_0} \chi(x)\\
& =&\frac{1}{N}\sum_{\chi \in C_0^\perp
}g(\bar{\chi})\chi(\gamma^a),
\end{eqnarray*}
where $C_0^\perp$ is the subgroup of $\widehat{\mathbb{F}_q^*}$ consisting of all characters $\chi$ which are trivial on $C_0$, i.e. $C_0^{\perp}$ the unique subgroup of $\widehat{\mathbb{F}_q^*}$ of order $N$. The above computations give the relationship between Gauss periods and Gauss sums.

Assume that $N=p_1^m$, where $p_1$ is an odd prime and $p_1 \equiv 5\ (\bmod \ 8)$, and $p_1>5$. Let $p\neq p_1$ be a prime such that $[\mathbb{Z}_N^*: \langle p \rangle ]=4$. It follows that $\gcd(p-1, p_1)=1$. (This can be seen as follows. If $p\equiv 1\pmod {p_1}$, then by using Lemma 3 of \cite[p.~42]{irero} repeatedly, we obtain that $p^{p_1^{m-1}}\equiv 1\pmod {p_1^m}$, contradicting the assumptions that ${\rm ord}_{p_1^m}(p)=\frac{p_1^{m-1}(p_1-1)}{4}$ and $p_1>5$.) Therefore we have $\gcd(p(p-1), N)=1$. Define $f=ord_N (p)= \frac{1}{4} \phi(N) $ and $\ q=p^f$. Let $C_0,C_1,\ldots, C_{N-1}$ be the $N$-th cyclotomic classes of $\mathbb{F}_q$. Define
\begin{equation}\label{defD}
D=\cup_{i=0}^{p_1^{m-1}-1}C_i.
\end{equation}

Using $D$ as connection set, we construct the Cayley graph $Cay(\mathbb{F}_q, D).$
\begin{theorem}\label{main}
The Cayley graph $Cay(\mathbb{F}_q, D)$ is an undirected, simple, regular graph of valency $|D|$, and it has at most five distinct restricted eigenvalues.
\end{theorem}
\begin{proof}
Note that $-1 \in C_0$ since either $2N |(q-1)$ or $q$ is even. Hence $-C_i=C_i$ for all $0\leq i \leq N-1$, so $D=-D$. Also $0\not \in D$. We conclude that the Cayley graph $Cay(\mathbb{F}_q,D)$ is undirected and without loops. The Cayley graph   $Cay(\mathbb{F}_q,D)$ is clearly regular of valency $|D|$. The restricted eigenvalues of $Cay(\mathbb{F}_q,D)$, as explained in \cite[p.~122]{bh}, are given by
\begin{equation*}  \ \psi(\gamma^a D)=\sum_{x\in D}\psi(\gamma^a x), \ 0\leq a \leq N-1.\end{equation*}

Now we turn to the computations of $\psi(\gamma^a D)$. We have
\begin{eqnarray*}
\psi(\gamma^a D) &= & \sum_{i=0}^{p_1^{m-1}-1}\psi(\gamma^a C_i) \\
%& = &  \sum_{i=0}^{p_1^{m-1}-1}\sum_{x\in C_{i+a}}\psi(x) \\
& = &  \sum_{i=0}^{p_1^{m-1}-1}\tau_{i+a} \\
 & = & \frac{1}{N} \sum_{i=0}^{p_1^{m-1}-1} \sum_{\chi \in C_0^\perp} g\left(\bar{\chi}\right)\chi(\gamma^{a+i}) \\
 & = &   \frac{1}{N}  \sum_{\chi \in C_0^\perp} g\left(\bar{\chi}\right) \sum_{i=0}^{p_1^{m-1}-1} \chi(\gamma^{a+i}). \\
\end{eqnarray*}

Consider the inner sum $\sum_{i=0}^{p_1^{m-1}-1} \chi(\gamma^{a+i})$, where $\chi \in C_0^\perp$. Note that $C_0^\perp$ is the unique subgroup of $\widehat{\mathbb{F}_q^*}$ of order $N=p_1^m$. If $\chi \in C_0^\perp$ and $ord(\chi)=1$ (that is, $\chi=\chi_0$), then $g(\bar{\chi})=-1$ and $\sum_{i=0}^{p_1^{m-1}-1} \chi(\gamma^{a+i})=p_1^{m-1}$. If $\chi \in C_0^\perp$ and $ord(\chi)=p_1^j\ (1\leq j \leq m-1)$, then $\chi(\gamma)\not =1$, $\chi(\gamma)^{p_1^{m-1}}=1$, and $\sum_{i=0}^{p_1^{m-1}-1} \chi(\gamma^{a+i})=\chi(\gamma^a)\sum_{i=0}^{p_1^{m-1}-1} \chi(\gamma^{i})=\chi(\gamma^a)\frac{\chi(\gamma)^{p_1^{m-1}}-1}{\chi(\gamma)\ -1}=0$. Hence,
\begin{eqnarray*}
\psi(\gamma^a D) = \frac{1}{N}\left(-p_1^{m-1} +\sum_{{\chi \in C_0^\perp} \atop{ ord(\chi)=p_1^m}} g\left(\bar{\chi}\right) \sum_{i=0}^{p_1^{m-1}-1} \chi(\gamma^{a+i}) \right).
\end{eqnarray*}

Next, we consider the characters $\chi \in C_0^\perp$ such that $ord(\chi)=N=p_1^m$, i.e., the generators of $C_0^\perp$. We define a multiplicative character $\theta$ of $\mathbb{F}_q$ by setting $\theta(\gamma)=\zeta_N$. It is clear that $\theta$ is a generator of $C_0^\perp$. Thus all generators of $C_0^\perp$ are given by $\theta^t$, where $t\in \mathbb{Z}_N^*$. It follows that

\begin{eqnarray*}
\psi(\gamma^a D) &= & \frac{1}{N}\big(-p_1^{m-1} +\sum_{{\chi \in C_0^\perp} \atop{ ord(\chi)=p_1^m}} g\big(\bar{\chi}\big) \sum_{i=0}^{p_1^{m-1}-1} \chi(\gamma^{a+i}) \big)\\
&=&\frac{1}{N}\big(-p_1^{m-1}+ \sum_{t \in \mathbb{Z}_{p_1^m}^*} g\big(\bar{\theta}^t\big) \sum_{i=0}^{p_1^{m-1}-1} \theta^t (\gamma^{a+i})\big).
\end{eqnarray*}

For convenience,  we set $$S_a:= \sum_{t \in \mathbb{Z}_{p_1^m}^*} g\big(\bar{\theta}^t\big) \sum_{i=0}^{p_1^{m-1}-1} \theta^t (\gamma^{a+i}),$$ where $0\leq a \leq N-1.$

For each $t\in \mathbb{Z}_{p_1^m}^*$, we write $t=t_1+p_1 t_2$, where $t_1 \in \mathbb{Z}_{p_1}^*, t_2\in \mathbb{Z}_{p_1^{m-1}}. $ For each $a, \, 0\leq a \leq N-1$, there is a unique $i_a\in \{0,1,2,\ldots, p_1^{m-1}-1\}$, such that $p_1^{m-1} \mid (a+i_a)$. Write $a+i_a=p_1^{m-1}j_a$ for some integer $j_a$. (When $N=p_1$, we have $i_a=0$ and $j_a=a$ for all $0\leq a\leq N-1$.)

By Theorem 2.3, we have $p^{-\frac{f-\tilde{f}}{2}-b } g\big(\bar {\theta}\big) \in O_K$. We can write $p^{-\frac{f-\tilde{f}}{2}-b } g\big(\bar{\theta}\big)= N_0\eta_0 + N_1 \eta_1+ N_2 \eta_2 +N_3 \eta _3$, $N_i\in \mathbb{Z}$, $\forall i$. Making the following transformation,
\begin{equation*}
\left \{\begin{array}{lll}
     M_0=N_0+N_1+N_2+N_3,\\
     M_1=N_0+N_1-N_2-N_3,\\
     M_2=N_0-N_1+N_2-N_3,\\
     M_3=N_0-N_1-N_2+N_3.
                 \end{array} \right.
\end{equation*}
By Theorem 2.4,  the integers $M_0,M_1,M_2,M_3$ satisfy the following conditions:
 \begin{equation}\label{Mequations}
 \qquad  \qquad  \left \{\begin{array}{lll}
    16 p^{\tilde{f}-2b}=M_0^2+p_1(M_1^2+M_2^2+M_3^2),\\
    2M_0M_2+2AM_1M_3=B(M_1^2-M_3^2),\\
    M_0+M_1+M_2+M_3 \equiv 0 \ (\bmod \ 4),\\
    M_1 \equiv M_2 \equiv M_3\  (\bmod \ 2),\\
    M_0\equiv 4p^{-b}\ (\bmod p_1).\\
      \end{array} \right.
\end{equation}
Here the notation is the same as in Section 2.

Next we want to determine how many distinct values $\psi(\gamma^aD)$, $0\leq a \leq N-1$, will take. Since $\psi(\gamma^a D)=\frac{1}{N} (-p_1^{m-1}+S_a)$, it suffices to determine the value distribution of  $\{S_a \mid 0\leq a \leq N-1\}$.

Since $\eta_j,\ 0\leq j\leq 3$, are in $\mathbb{Q}(\zeta_{p_1})$, we have $\sigma_t(\eta_j)=\sigma_{t_1+p_1t_2}(\eta_j)=\sigma_{t_1}(\eta_j)$. Hence $\sigma_t(g\big(\bar{\theta}\big))=\sigma_{t_1}(g\big(\bar{\theta}\big))$. Therefore $g\big({\bar{\theta}}^t\big)=g\big({\bar{\theta}}^{t_1}\big)=p^{\frac{f-\tilde{f}}{2}+b }(N_0\eta_{0}^{\sigma_{t_1}}  + N_1 \eta_{1}^{\sigma_{t_1}}+ N_2 \eta_{2}^{\sigma_{t_1}} +N_3 \eta _{3}^{\sigma_{t_1}})$. We now continue the computations of $S_a$. We have
 \begin{eqnarray*}
 S_a & =&\sum_{t \in \mathbb{Z}_{p_1^m}^*} g\big(\bar{\theta}^t\big) \sum_{i=0}^{p_1^{m-1}-1} \theta^t (\gamma^{a+i})\\
& =& \sum_{t_1 \in \mathbb{Z}_{p_1}^*} \sum_{t_2 \in \mathbb{Z}_{p_1^{m-1}}}  g\big(\bar{\theta}^{t_1+p_1 t_2}\big) \sum_{i=0}^{p_1^{m-1}-1} \theta^{t_1+p_1t_2} (\gamma^{a+i})\\
& =& \sum_{t_1 \in \mathbb{Z}_{p_1}^*} \sum_{t_2 \in \mathbb{Z}_{p_1^{m-1}}}  g\big(\bar{\theta}^{t_1}\big) \sum_{i=0}^{p_1^{m-1}-1} \theta^{t_1+p_1t_2} (\gamma^{a+i})\\
  &= &\sum_{t_1 \in \mathbb{Z}_{p_1}^*} \sum_{i=0}^{p_1^{m-1}-1}g\big(\bar{\theta}^{t_1}\big) \theta^{t_1}(\gamma^{a+i})  \sum_{t_2 \in \mathbb{Z}_{p_1^{m-1}}} (\theta^{p_1} (\gamma^{a+i}))^{t_2}.
\end{eqnarray*}
If $\theta^{p_1(a+i)}(\gamma)\not =1$, that is, $p_1^{m-1}\nmid (a+i)$,  then $$\sum_{t_2 \in \mathbb{Z}_{p_1^{m-1}}} (\theta^{p_1} (\gamma^{a+i}))^{t_2}=\frac{1-\theta^{p_1(a+i)\cdot p_1^{m-1}}(\gamma)}{1-\theta^{p_1(a+i)}(\gamma)}=0.  $$
Recall that for each $a, \ 0\leq a \leq N-1$, there is a unique $i_a\in \{0,1,2,\ldots, p_1^{m-1}-1\} $, such that $p_1^{m-1}\mid (a+i_a)$, and we write $a+i_a=p_1^{m-1}j_a$. Thus we have
%\begin{eqnarray*}
%\sum_{t_2 \in \mathbb{Z}_{p_1^{m-1}}} (\^{p_1} (\gamma^{a+i}))^{t_2}=\sum_{t_2 \in %\mathbb{Z}_{p_1^{m-1}}} (\chi^{{p_1}^{m-1}j_a} (\gamma))^{t_2}=p_1^{m-1}.
%\end{eqnarray*}
$$S_a = p_1^{m-1}\sum_{t_1 \in \mathbb{Z}_{p_1}^*} g\big(\bar{\theta}^{t_1}\big) \theta^{t_1}(\gamma^{p_1^{m-1}j_a}).$$
Note that by the definition of $\theta$, we have $\theta^{t_1}( \gamma^{p_1^{m-1}j_a})=\zeta_N^{p_1^{m-1}j_a\cdot t_1}=\zeta_{p_1}^{j_a\cdot t_1}.$
It will be convenient to introduce $\psi_{j_a}$, which is an additive character of the prime field $\mathbb{Z}_{p_1}$ such that $\psi_{j_a}(t_1)=\zeta_{p_1}^{j_a\cdot t_1}$. In this way, we have
$\theta^{t_1}( \gamma^{p_1^{m-1}j_a})=\psi_{j_a}(t_1)$. We now have
 \begin{eqnarray*}
S_a  & = &  p_1^{m-1}\sum_{t_1 \in \mathbb{Z}_{p_1}^*} g\big(\bar{\theta}^{t_1}\big) \psi_{j_a}(t_1)\\
 &= &p_1^{m-1} p^{\frac{f-\tilde{f}}{2}+b }  \sum_{t_1 \in \mathbb{Z}_{p_1}^*} (N_0\eta_{0}^{\sigma_{t_1}}  + N_1 \eta_{1}^{\sigma_{t_1}}+ N_2 \eta_{2}^{\sigma_{t_1}} +N_3 \eta _{3}^{\sigma_{t_1}}) \psi_{j_a}(t_1)\\
  &= &p_1^{m-1} p^{\frac{f-\tilde{f}}{2}+b }  \sum_{i=0}^{3}\sum_{t_1 \in g^i\langle p\rangle } (N_0\eta_{0}^{\sigma_{t_1}}  + N_1 \eta_{1}^{\sigma_{t_1}}+ N_2 \eta_{2}^{\sigma_{t_1}} +N_3 \eta _{3}^{\sigma_{t_1}}) \psi_{j_a}(t_1)\\
 & = & p_1^{m-1} p^{\frac{f-\tilde{f}}{2}+b } \big[ (N_0\eta_{0}+ N_1 \eta_{1}+ N_2 \eta_{2}+N_3 \eta _{3})  \sum_{t_1 \in \langle p\rangle}\psi_{j_a}(t_1)+ \\
& & \qquad \qquad\ \ (N_0\eta_{1}+ N_1 \eta_{2}+ N_2 \eta_{3}+N_3 \eta _{0})  \sum_{t_1 \in g\langle p\rangle}\psi_{j_a}(t_1)+ \\
 &  & \qquad \qquad\ \ (N_0\eta_{2}+ N_1 \eta_{3}+ N_2 \eta_{0}+N_3 \eta _{1})  \sum_{t_1 \in g^2\langle p\rangle}\psi_{j_a}(t_1)+\\
 &  & \qquad \qquad\ \ (N_0\eta_{3}+ N_1 \eta_{0}+ N_2 \eta_{1}+N_3 \eta _{2})  \sum_{t_1 \in g^3\langle p\rangle}\psi_{j_a}(t_1) \big ].
 \end{eqnarray*}

When $a$ runs through $\mathbb{Z}_N$, $j_a$ runs through $\mathbb{Z}_{p_1}$ correspondingly. Note that $\mathbb{Z}_{p_1}^*=\langle p\rangle  \cup g \langle p\rangle \cup g^2\langle p\rangle \cup g^3 \langle p\rangle.$ We therefore have five cases to consider according to $j_a=0$, and $j_a\in g^i\langle p\rangle$, $i=0,1,2,3$.\\

 \item{Case I }. $j_a=0$.  In this case, we have $\sum_{t_1 \in g^i\langle p\rangle}\psi_{j_a}(t_1)=\frac{p_1-1}{4},$ for $ 0\leq i\leq 3$.
 \begin{eqnarray*}
S_a & = &p_1^{m-1} p^{\frac{f-\tilde{f}}{2}+b } \big[ (N_0\eta_{0}+ N_1 \eta_{1}+ N_2 \eta_{2}+N_3 \eta _{3})  \frac{p_1-1}{4}+ \\
& & \qquad \qquad\ \ (N_0\eta_{1}+ N_1 \eta_{2}+ N_2 \eta_{3}+N_3 \eta _{0})  \frac{p_1-1}{4}+ \\
 &  & \qquad \qquad\ \ (N_0\eta_{2}+ N_1 \eta_{3}+ N_2 \eta_{0}+N_3 \eta _{1})  \frac{p_1-1}{4}+\\
 &  & \qquad \qquad\ \ (N_0\eta_{3}+ N_1 \eta_{0}+ N_2 \eta_{1}+N_3 \eta _{2})  \frac{p_1-1}{4} \big ]\\
&=& -p_1^{m-1} p^{\frac{f-\tilde{f}}{2}+b } (N_0+ N_1 + N_2+N_3)\frac{p_1-1}{4}.\\
\end{eqnarray*}
This value of $S_a$ will be denoted by $p_1^{m-1} p^{\frac{f-\tilde{f}}{2}+b }T_1$, where $T_1=(N_0+ N_1 + N_2+N_3)\frac{1-p_1}{4}$.

\item{Case II}. $j_a\in \langle p\rangle$. In this case $\sum_{t_1 \in g^i\langle p\rangle}\psi_{j_a}(t_1)=\eta_i,$ $0\leq i\leq 3$. We have
\begin{eqnarray*}
S_a &= & p_1^{m-1} p^{\frac{f-\tilde{f}}{2}+b } \big[ (N_0\eta_{0}+ N_1 \eta_{1}+ N_2 \eta_{2}+N_3 \eta _{3})\eta_0 + \\
&&\qquad \qquad \  (N_0\eta_{1}+ N_1 \eta_{2}+ N_2 \eta_{3}+N_3 \eta _{0}) \eta_1 + \\
&& \qquad \qquad \  (N_0\eta_{2}+ N_1 \eta_{3}+ N_2 \eta_{0}+N_3 \eta _{1}) \eta_2 + \\
 && \qquad \qquad \  (N_0\eta_{3}+ N_1 \eta_{0}+ N_2 \eta_{1}+N_3 \eta _{2})\eta_3 \big] \\
&= &p_1^{m-1} p^{\frac{f-\tilde{f}}{2}+b } \big[N_0(\eta_0^2+\eta_1^2+\eta_2^2+\eta_3^2)+ \\ &&\qquad \qquad \  N_1(\eta_0\eta_{1}+\eta_1\eta_{2}+\eta_2\eta_{3}+\eta_3\eta_{0}) + \\ &&\qquad \qquad \  N_2(\eta_0\eta_{2}+\eta_1\eta_{3}+\eta_2\eta_{0}+\eta_3\eta_{1})+ \\
&&\qquad \qquad \  N_3(\eta_0\eta_{3}+\eta_1\eta_{0}+\eta_2\eta_{1}+\eta_3\eta_{2}) \big ].\\
 \end{eqnarray*}
This value of $S_a$  will be denoted by $p_1^{m-1} p^{\frac{f-\tilde{f}}{2}+b }T_2$.

\item{Case III}. $j_a\in g\langle p\rangle$. In this case $\sum_{t_1 \in g^i\langle p\rangle}\psi_{j_a}(t_1)=\eta_{i+1}, $ $ 0\leq i\leq 3$. Similarly we have
\begin{eqnarray*}
S_a &= &p_1^{m-1} p^{\frac{f-\tilde{f}}{2}+b } \big[ N_0(\eta_0\eta_{1}+\eta_1\eta_{2}+\eta_2\eta_{3}+\eta_3\eta_{0})+\\
&& \qquad \qquad \ N_1(\eta_0^2+\eta_1^2+\eta_2^2+\eta_3^2) +\\
 && \qquad \qquad \   N_2(\eta_0\eta_{1}+\eta_1\eta_{2}+\eta_2\eta_{3}+\eta_3\eta_{0})+\\
 && \qquad \qquad \    N_3(\eta_0\eta_{2}+\eta_1\eta_{3}+\eta_2\eta_{0}+\eta_3\eta_{1})\big].\\
 \end{eqnarray*}
This value of $S_a$ will be denoted by $p_1^{m-1} p^{\frac{f-\tilde{f}}{2}+b }T_3$.

\item{Case IV}. $j_a\in g^2\langle p\rangle$. In this case $\sum_{t_1 \in g^i\langle p\rangle}\psi_{j_a}(t_1)=\eta_{i+2},$ $0\leq i\leq 3$. Similarly we have
\begin{eqnarray*}
S_a &= & p_1^{m-1} p^{\frac{f-\tilde{f}}{2}+b } \big[ N_0(\eta_0\eta_{2}+\eta_1\eta_{3}+\eta_2\eta_{0}+\eta_3\eta_{1})+\\
&& \qquad \qquad \   N_1(\eta_0\eta_{1}+\eta_1\eta_{2}+\eta_2\eta_{3}+\eta_3\eta_{0}) +\\
&& \qquad \qquad \    N_2(\eta_0^2+\eta_1^2+\eta_2^2+\eta_3^2)+\\
&& \qquad \qquad \     N_3(\eta_0\eta_{3}+\eta_1\eta_{0}+\eta_2\eta_{1}+\eta_3\eta_{2})\big].\\
\end{eqnarray*}
This value of $S_a$ will be denoted by $p_1^{m-1} p^{\frac{f-\tilde{f}}{2}+b }T_4$.

\item{Case V}. $j_a\in g^3\langle p\rangle$. In this case $\sum_{t_1 \in g^i\langle p\rangle}\psi_{j_a}(t_1)=\eta_{i+3},$ $ 0\leq i\leq 3$. Similarly we have
\begin{eqnarray*}
S_a &= &p_1^{m-1} p^{\frac{f-\tilde{f}}{2}+b } \big[ N_0(\eta_0\eta_{3}+\eta_1\eta_{0}+\eta_2\eta_{1}+\eta_3\eta_{2})+\\
 && \qquad \qquad \     N_1(\eta_0\eta_{2}+\eta_1\eta_{3}+\eta_2\eta_{0}+\eta_3\eta_{1}) +\\
&& \qquad \qquad \      N_2(\eta_0\eta_{1}+\eta_1\eta_{2}+\eta_2\eta_{3}+\eta_3\eta_{0})+\\
&& \qquad \qquad \       N_3(\eta_0^2+\eta_1^2+\eta_2^2+\eta_3^2) \big].\\
\end{eqnarray*}
This value of $S_a$ will be denoted by $p_1^{m-1} p^{\frac{f-\tilde{f}}{2}+b }T_5$.

Therefore we have shown that $S_a$, $0\leq a\leq N-1$, take at most five distinct values. It follows that the Cayley graph $Cay(\mathbb{F}_q,D)$ has at most five distinct restricted eigenvalues. The proof of the theorem is complete.
\end{proof}

We are now ready to consider the question that under what conditions, the Cayley graph $Cay(\mathbb{F}_q,D)$, with $D$ defined in (\ref{defD}), is strongly regular. By Theorem~\ref{basicthm}, the question is the same as asking under what conditions, the Cayley graph $Cay(\mathbb{F}_q,D)$ will have exactly two distinct restricted eigenvalues. Using the transformation between $\{N_0,N_1,N_2,N_3\}$ and $\{M_0,M_1,M_2,M_3\}$, and the following equations satisfied by $\eta_i$,
\begin{equation*}
\left\{\begin{array}{lll}
    \eta_0^2+\eta_1^2+\eta_2^2+\eta_3^2=\frac{1-p_1}{4},\\
   \eta_0\eta_1+\eta_1\eta_2+\eta_2\eta_3+\eta_3\eta_0=\frac{1-p_1}{4},\\   \eta_0\eta_2+\eta_1\eta_3+\eta_2\eta_0+\eta_3\eta_1=\frac{1+3p_1}{4},\\
   \end{array} \right.
\end{equation*}
we have $\{T_1,T_2,T_3,T_4,T_5\}=\{\frac{1-p_1}{4}M_0,\  \frac{1-p_1}{4}M_0+p_1N_0,\  \frac{1-p_1}{4}M_0+p_1N_1,\  \frac{1-p_1}{4}M_0+p_1N_2, \ \frac{1-p_1}{4}M_0+p_1N_3\}.$  From the proof of Theorem~\ref{main}, we see that the value distribution of the restricted eigenvalues of $Cay(\mathbb{F}_q,D)$ is completely determined by the value distribution of $\{T_1,T_2,T_3,T_4,T_5\}$.

\begin{theorem}\label{2ndthm}
If $Cay(\mathbb{F}_q,D)$ is strongly regular, then either $p_1-1$ or $p_1-9$ is a perfect square. In the case where $p_1-1$ is a square, $Cay(\mathbb{F}_q,D)$ is strongly regular if and only if the integer solutions $(M_0,M_1,M_2,M_3)$ of (\ref{Mequations}) satisfy $(M_0:M_1:M_2:M_3)\in\{\ (1:1:1:1),\ (1:1:-1:-1),\ (1:-1:1:-1),\ (1:-1:-1:1)\ \}$. In the case where $p_1-9$ is a square, $Cay(\mathbb{F}_q,D)$ is strongly regular if and only if the integer solutions $(M_0,M_1,M_2,M_3)$ of (\ref{Mequations}) satisfy $(M_0:M_1:M_2:M_3)\in\{\ (3:-1:-1:-1),\ (3:-1:1:1),\ (3:1:-1:1),\ (3:1:1:-1)\}$.
\end{theorem}
\begin{proof}
Up to a permutation of indices, we may assume that
\begin{equation*}
\left\{\begin{array}{lll}
T_1=\frac{1-p_1}{4}M_0,\\
 T_2=\frac{1-p_1}{4}M_0+p_1N_0,\\
 T_3= \frac{1-p_1}{4}M_0+p_1N_1,\\
 T_4=\frac{1-p_1}{4}M_0+p_1N_2, \\
 T_5=\frac{1-p_1}{4}M_0+p_1N_3. \end{array} \right.
\end{equation*}
We first note that the set $\{T_1,T_2,T_3,T_4,T_5\}$ has at least two distinct elements. Otherwise, we will have $N_0=N_1=N_2=N_3=0$; it follows that the Gauss sum $g\big({\bar \theta}\big)=0$, which is impossible.

If the set $\{T_1,T_2,T_3,T_4,T_5\}$ has exactly two distinct elements, there are fifteen possible cases in total. We discuss these cases one by one.\\

\item{Case 1}. $T_2=T_3=T_4=T_5 \not =T_1  \Leftrightarrow N_0=N_1=N_2=N_3\not =0 \Leftrightarrow M_1=M_2=M_3=0,M_0\not =0 $. Under the assumptions of this case, we have $M_0^2=16p^{\tilde{f} -2b}$. But $\tilde{f}=\frac {p_1 -1}{4}$ is odd since $p_1\equiv 5$ (mod 8). It follows that $M_0\not\in \mathbb{Z}$, a contradiction. We conclude that Case 1 cannot occur.\\

\item{Case 2}. $T_1=T_3 = T_4=T_5 \not= T_2 \Leftrightarrow N_1=N_2=N_3=0,\  N_0\not =0  \Leftrightarrow (M_0:M_1:M_2:M_3)=(1:1:1:1)$. In this case we have $A=-1$ and $p_1-1=B^2$.\\

\item{Case 3}. $T_1=T_2 = T_4=T_5 \not= T_3 \Leftrightarrow N_0=N_2=N_3=0,\ N_1\not =0 \Leftrightarrow (M_0:M_1:M_2:M_3)=(1:1:-1:-1) $. In this case we have $A=-1$ and $p_1-1=B^2$.\\

\item{Case 4}. $T_1=T_2 = T_3 = T_5 \not= T_4 \Leftrightarrow N_0=N_1=N_3=0,\ N_2\not =0 \Leftrightarrow (M_0:M_1:M_2:M_3)=(1:-1:1:-1)$. In this case we have $A=-1$ and $p_1-1=B^2$.\\

\item{Case 5}. $T_1=T_2=T_3 = T_4 \not =T_5 \Leftrightarrow N_0=N_1=N_2=0,\ N_3\not =0 \Leftrightarrow (M_0:M_1:M_2:M_3)=(1:-1:-1:1)$. In this case we have $A=-1$ and $p_1-1=B^2$.\\

\item{Case 6}. $T_1= T_4=T_5 \not= T_2=T_3 \Leftrightarrow N_2=N_3=0,\ N_0=N_1\not =0 \Leftrightarrow M_0=M_1\neq 0,\;M_2=M_3=0$. In this case we have $B=0$, which is impossible.\\

\item{Case 7}. $T_1= T_3=T_5 \not= T_2=T_4 \Leftrightarrow N_1=N_3=0,\ N_0=N_2\not =0 \Leftrightarrow M_0=M_2,M_1=M_3=0$. In this case, we have $M_0=M_1=M_2=M_3=0$, which is impossible.\\

\item{Case 8}. $T_1= T_3=T_4 \not= T_2=T_5 \Leftrightarrow N_1=N_2=0,\ N_0=N_3\not =0\Leftrightarrow M_0=M_3\neq 0\;,M_1=M_2=0$. In this case we have $B=0$, which is impossible.\\

\item{Case 9}. $T_1= T_2=T_5 \not= T_3=T_4 \Leftrightarrow N_0=N_3=0,\ N_1=N_2\not =0\Leftrightarrow M_0=-M_3,M_1=M_2=0$. In this case we have $B=0$, which is impossible.\\

\item{Case 10}. $T_1= T_2=T_4 \not= T_3=T_5 \Leftrightarrow N_0=N_2=0,\ N_1=N_3\not =0\Leftrightarrow M_0=-M_2,M_1=M_3=0$. In this case we have $M_0=M_1=M_2=M_3=0$, which is impossible.\\

\item{Case 11}. $T_1= T_2=T_3 \not =T_4=T_5 \Leftrightarrow N_0=N_1=0,\ N_2=N_3\not =0\Leftrightarrow M_0=-M_1,M_2=M_3=0$. In this case we have $B=0$, which is impossible.\\

\item{Case 12}. $T_3= T_4=T_5 \not= T_1=T_2 \Leftrightarrow N_1=N_2=N_3\not =0,\ N_0 =0\Leftrightarrow (M_0:M_1:M_2:M_3)=(3:-1:-1:-1)$. In this case we have $A=3$ and $p_1-9=B^2$.\\

\item{Case 13}. $T_2= T_4=T_5 \not =T_1=T_3 \Leftrightarrow N_0=N_2=N_3\not =0,\ N_1=0\Leftrightarrow (M_0:M_1:M_2:M_3)=(3:-1:1:1)$. In this case we have $A=3$ and $p_1-9=B^2$.\\

\item{Case 14}. $T_2=T_3= T_5 \not =T_1=T_4 \Leftrightarrow N_0=N_1=N_3\not =0,\ N_2=0\Leftrightarrow (M_0:M_1:M_2:M_3)=(3:1:-1:1)$. In this case we have $A=3$ and $p_1-9=B^2$.\\

\item{Case 15}. $T_2=T_3= T_4\not =T_1=T_5 \Leftrightarrow N_0=N_1=N_2\not =0,\ N_3 =0\Leftrightarrow (M_0:M_1:M_2:M_3)=(3:1:1:-1)$. In this case we have $A=3$ and $p_1-9=B^2$.\\

If $Cay(\mathbb{F}_q,D)$ is strongly regular, then it has exactly two distinct restricted eigenvalues, thus $\{T_1,T_2,T_3,T_4,T_5\}$ has exactly two distinct elements. From the analysis above, either $p_1-1$ or $p_1-9$ is a square;  suppose $(M_0,M_1,M_2,M_3)$ is a solution of (\ref{Mequations}), we see that $(M_0,M_1,M_2,M_3)$ must be one of the possibilities listed in the statement of the theorem. That is, when $A=-1,\ p_1-1$ is a perfect square, $(M_0:M_1:M_2:M_3)\in\{\ (1:1:1:1),\ (1:1:-1:-1),\ (1:-1:1:-1),\ (1:-1:-1:1)\ \}$;  when $A=3,\ p_1-9$ is perfect square and $(M_0:M_1:M_2:M_3)\in\{\ (3:-1:-1:-1),\ (3:-1:1:1),\ (3:1:-1:1),\ (3:1:1:-1)\}$.

Conversely, if the integer solutions $(M_0,M_1,M_2,M_3)$ of (\ref{Mequations}) satisfy the conditions stated in the theorem, then it is easy to see from the above analysis that $\{T_1,T_2,T_3,T_4,T_5\}$ has exactly two distinct elements. It follows that $Cay(\mathbb{F}_q,D)$ is strongly regular.

The proof of the theorem is now complete.
\end{proof}

\section{New infinite families of strongly regular Cayley graphs}

We used a computer to search for prime pairs $(p, p_1)$, $2\leq p <10,000$, $3\leq p_1<10,000$, satisfying the conditions specified in Section 2 and in the statement of Theorem~\ref{2ndthm}. We found two such pairs which are given below. Note that in general for a prime pair $(p, p_1)$ satisfying the conditions $p_1\equiv 5\pmod 8$, $\gcd(p(p-1), p_1)=1$ and ${\rm ord}_{p_1^m}(p)=\phi(p_1^m)/4$ for all $m\geq 1$, there are possibly many solutions $(M_1,M_2,M_3,M_4)$  to (\ref{Mequations}); only those solutions $(M_1,M_2,M_3,M_4)$ which can be used to represent the Gauss sums $g\big({\bar \theta}\big)$ should be considered. We refer the reader to Lemma 3.2 of \cite{Feng} for a method to decide when a solution $(M_1,M_2,M_3,M_4)$ to (\ref{Mequations}) can be used to represent the Gauss sum $g\big({\bar \theta}\big)$.

\vspace{0.1in}

 \noindent {\bf Example 4.1.}
Let $p_1=37,\;  p=7,\; N=p_1^m$ where $m\geq 1$ is any integer. Note that in this case we have $p_1\equiv 5$ (mod 8) and $p_1>5$. It is straightforward to check that $ord_{37}(7)=9=\frac{\phi(37)}{4}$. By induction on $m$, one can show that $ord_{37^m}(7)=\frac{\phi(37^m)}{4}$. Let $f=ord_{37^m}(7)=\frac{\phi(37^m)}{4}$ and $\mathbb{F}_q$ be the finite field of order $q=7^f$. Let $\gamma$ be a fixed primitive element of $\mathbb{F}_q$. Let $C_0=\langle \gamma^N\rangle, C_1=\gamma C_0,\ldots, C_{N-1}=\gamma^{N-1}C_0$ be the $N$-th cyclotomic classes of $\mathbb{F}_q$ and let $$D=\bigcup _{i=0}^{37^{m-1}-1}C_i.$$ We claim that the Cayley graph $Cay(\mathbb{F}_q,D)$ is strongly regular. To prove this claim, it suffices to apply Theorem~\ref{2ndthm} to the current situation.

 \begin{lemma}{\em (Example 1, \cite{Feng})}
 When $p_1=13$ or $37$, we have $$b={\rm min}\{b_0,b_1,b_2,b_3\}=\frac{\tilde{f}-1}{2},$$ where $\tilde{f}=\frac{\phi(p_1)}{4}$.
  \end{lemma}

Now for $p_1=37$, we have $\tilde{f}=\frac{\phi(37)}{4}=9,\ b=4$, and $p_1-1=36$ is a perfect square. The integer solutions $(A,B)$ to $p_1=A^2+B^2$ with $A\equiv 3\ (\bmod\  4)$ are $(-1, \pm 6)$. That is, $A=-1$ and $B=\pm 6$. Also $4p^{-b}=4\cdot 7^{-4}\equiv 4\cdot 9\equiv -1\ (\bmod\  37)$. We need to determine the $(M_0,M_1,M_2,M_3)$ satisfying (\ref{Mequations}). In our case, (\ref{Mequations}) becomes:
\begin{equation*}
\left\{\begin{array}{lll}
    112=M_0^2+37(M_1^2+M_2^2+M_3^2),\\
    2M_0M_2-2M_1M_3=B(M_1^2-M_3^2),\\
    M_0+M_1+M_2+M_3 \equiv 0 \ (\bmod \ 4),\\
    M_1 \equiv M_2 \equiv M_3 \ (\bmod \ 2),\\
    M_0\equiv -1\ (\bmod \ 37).\\
        \end{array} \right.
\end{equation*}
From the first equation we obtain $M_0^2=1$ and $M_1^2+M_2^2+M_3^2=3$. Therefore, $M_0=-1$, and $M_1,M_2,M_3\in \{\pm 1\}$. Together with the conditions, we get a total of four integer solutions $(-1,1,1,-1),\ (-1,1,-1,1),\ (-1,-1,1,1),\ (-1,-1,-1,-1)$.  Since each of these four solutions satisfies the conditions of Theorem 3.2, we conclude that $Cay(\mathbb{F}_q,D)$ is a strongly regular graph, with parameters
$$v=7^{9\cdot 37^{m-1}}, \ k=\frac{v-1}{37}, \; r=\frac{9\cdot 7^{\frac{ 9\cdot 37^{m-1}  -1} {2} } -1}{37},\; {\rm and}\ s=\frac{-4 \cdot 7^{\frac{9\cdot 37^{m-1} +1} {2} } -1}{37}.$$\\

\noindent {\bf Example 4.2.} Let $p_1=13,\  p=3,\ N=p_1^m$, where $m\geq 1$ is an integer. By induction on $m$, we also can show that $ord_{13^m}(3)=\frac{\phi(13^m)}{4}.$ Also, we let $f=\frac{\phi(13^m)}{4},\ q=3^f$, and $C_0,C_1,\ldots, C_{N-1}$ be the $N$-th cyclotomic classes of $\mathbb{F}_q$. Using $$D=\bigcup _{i=0}^{13^{m-1}-1}C_i$$ as connection set, we construct the Cayley graph $Cay(\mathbb{F}_q,D)$. Now $p_1-9=4$ is a perfect square, $\tilde{f}=\frac{\phi(13)}{4}=3$ and $b=\frac{\tilde{f}-1}{2}=1$ by Lemma 4.1.

The integer solutions $(A,B)$ to $p_1=A^2+B^2$ with $A\equiv 3\ (\bmod\  4)$ are $(3, \pm 2)$. That is, $A=3$ and $B=\pm 2$. Also $4p^{-b}=4\cdot 3^{-1}\equiv 4\cdot (-4)\equiv -3\ (\bmod \ 13)$. We need to determine the $(M_0,M_1,M_2,M_3)$ satisfying (\ref{Mequations}). In our case, (\ref{Mequations}) becomes
\begin{equation*}
\left\{\begin{array}{lll}
    48=M_0^2+13(M_1^2+M_2^2+M_3^2),\\
    2M_0M_2+6M_1M_3=B(M_1^2-M_3^2),\\
    M_0+M_1+M_2+M_3 \equiv 0 \ (\bmod \ 4),\\
    M_1 \equiv M_2 \equiv M_3 \ (\bmod \ 2),\\
    M_0\equiv -3\ (\bmod \ 13).\\
        \end{array}\right.
\end{equation*}
From the first equation we obtain $M_0^2=9$ and $M_1^2+M_2^2+M_3^2=3$. Therefore, $M_0=-3$ and $M_1,M_2,M_3\in \{\pm 1\}$. Similarly, we also get four solutions $(-3,-1,-1,1),\ (-3,1,-1,-1),\ (-3,-1,1,-1),\ (-3,1,1,1)$. Since each of them satisfies the conditions of Theorem 3.2, we conclude that $Cay(\mathbb{F}_q,D)$ is also a strongly regular graph.

If $m=1$, then $N=13,\ f=3,\ q=p^f=27$ and $D=C_0=\mathbb{F}_3^*$, where $\mathbb{F}_3$ is the prime subfield of $\mathbb{F}_{3^3}$. The strongly regular graph in this case belongs to the so-called {\it subfield case}, and is rather boring. But for $m \geq 2$, the strongly regular graphs $Cay(\mathbb{F}_q,D)$ are new and their parameters are
$$v=3^{3\cdot 13^{m-1}},\ k=\frac{v-1}{13}, \; r=\frac{3^{ \frac{3\cdot 13^{m-1}+3 } {2}   }-1  }  {13},  \; {\rm and} \; s=\frac{-4\cdot 3^{ \frac{3\cdot 13^{m-1}-1 } {2}   }-1  }  {13}.$$


\begin{thebibliography}{99}
\bibitem{batdov} L. Batten, J. Dover, Some sets of type $(m,n)$ in cubic order planes, {\it Des. Codes Cryptogr.} {\bf 16} (1999), 211--213.

\bibitem{bmw} L. D. Baumert, W. H. Mills, R. L. Ward, Uniform cyclotomy, {\it J. Number Theory}, {\bf 14} (1982), 67--82.


\bibitem{bew}B. C. Berndt, R. J. Evans, and K. S. Williams, {\it Gauss and Jacobi
Sums}, A Wiley-Interscience Publication, 1998.

\bibitem{bh} A. E. Brouwer,  W. H. Haemers, {\it Spectra of Graphs}, Springer Universitext, 2012.

\bibitem{bwx} A. E. Brouwer, R. M. Wilson, and Q. Xiang, Cyclotomy and strongly regular graphs, {\it J. Algebraic Combin.} {\bf 10} (1999), 25--28.

\bibitem{CK86}R. Calderbank, W. M. Kantor, The geometry of two-weight codes, {\it
Bull. London Math. Soc} {\bf 18-2} (1986), 97--122.

\bibitem{vanDamMu} E. van Dam, M. Muzychuk, Some implications on amorphic association schemes,
{\it J. Combin. Theory} (A) {\bf 117} (2010), 111-127

\bibitem{Feng} K. Feng, J. Yang, and S. Luo, Gauss sum of index 4:(1) cyclic case, {\it Acta Math. Sin. (Engl. Ser.)} {\bf 21-6} (2005), 1425--1434.

\bibitem{fengxiang} T. Feng, Q. Xiang, Strongly regular graphs from unions of cyclotomic classes, {\it J. Combin. Theory} (B), in press.

\bibitem{cg} C. Godsil, G. Royle, {\it Algebraic Graph Theory}, GTM 207, Springer-Verlag, 2001.

\bibitem{ikutam} T. Ikuta, A. Munemasa, Pseudocyclic association schemes and strongly regular graphs, {\it European J. Combin.} {\bf 31} (2010), 1513--1519.

\bibitem{irero} K. Ireland, M. Rosen, {\it A Classical Introduction to Modern Number Theory, Second Edition},  Graduate Text in Math. No. 84, Springer-Verlag, Berlin/New York/Heidelberg, 2003.

\bibitem{Lang} P. Langevin, A new class of two-weight codes, in {\em Finite Fields and Applications} (Glasgow 1995), London Math. Soc. Lecture Note Series, No. 233, S. Cohen and H. Niederreiter, eds. Cambridge University Press, 1996, pp. 181--187.

%\bibitem{Lang} P. Langevin, Calculs de certaines sommes de Gauss, {\it J. Number Theory}  {\bf 63} (1997), 59--64.

\bibitem{del} C. L. M.  de Lange, Some new cyclotomic strongly regular graphs, {\it J. Algebraic Combin.}  {\bf 4} (1995), 329--330.

\bibitem{vanLintSch} J. H. van Lint, A. Schrijver, Construction of strongly regular graphs, two-weight codes and partial geometries by finite fields, {\it Combinatorica} {\bf 1} (1981), 63--73.

\bibitem{Ma94} S. L. Ma, A survey of partial difference sets, {\it Des. Codes Cryptogr.} {\bf 4} (1994), 221-261.

%\bibitem{Mbo} O. D. Mbodj, Quadratic Gauss sums, {\it Finite Fields and Appl.}  {\bf 4} (1998), 347--361.

\bibitem{McE} R. J. McEliece, Irreducible cyclic codes and Gauss sums. Combinatorics (Proc. NATO Advanced Study Inst., Breukelen, 1974), Part 1: Theory of designs, finite geometry and coding theory, pp. 179--196. Math. Centre Tracts, No. 55, Math. Centrum, Amsterdam.

\bibitem{schwh} B. Schmidt, C. White, All two-weight irreducible cyclic codes, {\it Finite Fields Appl.}  {\bf 8} (2002), 1--17.

\bibitem{stic} L. Stickelberger, \"Uber eine Verallgemeinerung der Kreistheilung, {\it Math. Annal.} {\bf 37} (1890), 321--367.

\bibitem{yx} J. Yang, L. Xia, Complete solving of explicit evaluation of Gauss sums in the index 2 case, {\it  Sci. China Math.} {\bf 53} (2010), 2525--2542.



\end{thebibliography}
\end{document}